\documentclass[twoside,11pt]{article}
\usepackage{natbib} 
\usepackage{graphics}
\usepackage{algorithmic}
\usepackage{algorithm}
\usepackage{amsmath,amsfonts,amssymb, amsthm}
\usepackage{comment}
\usepackage{multirow}
\usepackage{graphicx}
\usepackage{pifont}
\newcommand{\cmark}{\ding{51}}%
\newcommand{\xmark}{\ding{55}}%
\usepackage{url}
\usepackage{tikz}
\usetikzlibrary{arrows}

\newdimen\arrowsize
\pgfarrowsdeclare{arcsq}{arcsq}
{
  \arrowsize=0.2pt
  \advance\arrowsize by .5\pgflinewidth
  \pgfarrowsleftextend{-4\arrowsize-.5\pgflinewidth}
  \pgfarrowsrightextend{.5\pgflinewidth}
}
{
  \arrowsize=1.5pt
  \advance\arrowsize by .5\pgflinewidth
  \pgfsetdash{}{0pt} 
  \pgfsetroundjoin   
  \pgfsetroundcap    
  \pgfpathmoveto{\pgfpoint{0\arrowsize}{0\arrowsize}}
  \pgfpatharc{-90}{-140}{4\arrowsize}
  \pgfusepathqstroke
  \pgfpathmoveto{\pgfpointorigin}
  \pgfpatharc{90}{140}{4\arrowsize}
  \pgfusepathqstroke
}

\newcommand{\independent}{\mbox{${}\perp\mkern-11mu\perp{}$}}
\newcommand{\notindependent}{\mbox{${}\not\!\perp\mkern-11mu\perp{}$}}

\newcommand{\B}[1]{\mathbf{#1}}

\newcommand{\PA}[2][]{{\B{PA}}^{#1}_{#2}}

\newcommand{\ND}[2][]{{\B{ND}}^{#1}_{#2}}

\newcommand{\given}{\,|\,}

\newtheorem{corollary}{Corollary}
\newtheorem{lemma}{Lemma}
\newtheorem{example}{Example}
\newtheorem{definition}{Definition}
\newtheorem{proposition}{Proposition}

\graphicspath{{./}}

\begin{document}
\title{On the Intersection Property of Conditional Independence and its Application to Causal Discovery}
\author{Jonas Peters\\ peters@stat.math.ethz.ch\\
ETH Zurich} 

\maketitle
\begin{abstract}
This work investigates the intersection property of conditional independence. It states that for random variables $A,B,C$ and $X$ we have that $X \independent A \, \given \, B,C$ 
and $X \independent B \, \given \, A,C$ implies $X \independent (A,B) \, \given \,C$. Here, ``$\independent$'' stands for statistical independence. 
Under the assumption that the joint distribution has a continuous density, we provide necessary and sufficient conditions under which the intersection property holds. The result has direct applications to causal inference: it leads to strictly weaker conditions under which the graphical structure becomes identifiable from the joint distribution of an additive noise model.
\end{abstract}

\section{Introduction}
\subsection{Application to Causal Inference}
Inferring causal relationships is a major challenge in science. In the last decades considerable effort has been made in order to learn causal statements from observational data. Causal discovery methods make assumptions that relate the joint distribution with properties of the causal graph. Constraint-based or independence-based methods \citep{Pearl2009, Spirtes2000} and some score-based methods \citep{Chickering2002, Heckerman1999} assume the Markov condition and faithfulness. A distribution is said to be Markov with respect to a directed acyclic graph (DAG) $G$ if each d-separation in the graph implies the corresponding (conditional) independence; the distribution is faithful with respect to $G$ if the reverse statement holds. These two assumptions render the Markov equivalence class of the correct graph identifiable from the joint distribution, i.e. the skeleton and the v-structures of the graph can be inferred from the joint distribution \citep{Verma1991}.
Methods like LiNGAM \citep{Shimizu2006} or additive noise models \citep{Hoyer2008, Peters2013anm} assume the Markov condition, too, but do not require faithfulness; instead, these methods assume that the structural equations come from a restricted model class (e.g. linear with non-Gaussian noise or non-linear with additive Gaussian noise). In order to prove that the directed acyclic graph (DAG) is identifiable from the joint distribution \citet{Peters2013anm} require a strictly positive density. Their proof makes use of the intersection property of conditional independence (Definition~\ref{def:inters}) which is known to hold for positive densities \citep[e.g.][1.1.5]{Pearl2009}.

\subsection{Main Contributions}
In Section~\ref{sec:main} we provide a sufficient and necessary condition on the density for the intersection property to hold (Corollary~\ref{cor:1}). This result is of interest in itself since the developed condition is weaker than strict positivity.

As mentioned above, some causal discovery methods based on structural equation models require the intersection property for identification; they therefore rely on the strict positivity of the density. This can be achieved by fully supported noise variables, for example.
Using the new characterization of the intersection property we can now replace the condition of strict positivity. In fact, we show in Section~\ref{sec:appl} that noise variables with a path-connected support are sufficient for identifiability of the graph (Proposition~\ref{prop:pcnoise}).
This is already known for linear structural equation models \citep{Shimizu2006} but not for non-linear models. As an alternative, we provide a condition that excludes constant functions and leads to identifiability, too (Proposition~\ref{prop:nonconstfunct}).

In Section~\ref{sec:ex}, we provide an 
example of a structural equation model that violates the intersection property (but satisfies causal minimality).
Its corresponding graph is not identifiable from the joint distribution.
In correspondence to the theoretical results of this work, some noise densities in the example are do not have a path-connected support and the functions are partially constant. 
We are not aware of any causal discovery method that is able to infer the correct DAG or the correct Markov equivalence class; the example therefore shows current limits of causal inference techniques. It is non-generic in the case that it violates all sufficient assumptions mentioned in Section~\ref{sec:appl}.

\subsection{Conditional Independence and the Intersection Property} \label{sec:ciip}
We now formally introduce the concept of conditional independence in the presence of densities and the intersection property. 
Let therefore $A, B, C$ and $X$ be (possibly multi-dimensional) random variables that take values in metric spaces $\mathcal{A, B, C}$ and $\mathcal{X}$ respectively. We first introduce assumptions regarding the existence of a density and some of its properties that appear in different parts of this paper.
\begin{itemize}
\item[(A0)\ ] The distribution is absolutely continuous with respect to a product measure of a metric space. We denote the density by $p(\cdot)$. This can be a probability mass function or a probability density function, for example.
\item[(A1)\ ] The density $(a,b,c) \mapsto p(a,b,c)$ is continuous. 
\item[(A2)\ ] For each $c$ with $p(c)>0$ the set $\mathrm{supp}_c(A,B) := \{(a,b)\,:\, p(a,b,c)>0\}$ contains only one path-connected component (see Definition~\ref{def:uc}).
\item[(A2')] The density $p(\cdot)$ is strictly positive.
\end{itemize}
Condition (A2') implies (A2). We assume (A0) throughout the whole work.

In this paper we work with the following definition of conditional independence.
\begin{definition}[Conditional Independence] \label{def:ci}
We call $X$ independent of $A$ conditional on $B$ and write $X \independent A \given B$ if and only if
\begin{equation} \label{eq:defind}
p(x,a \given b) = p(x\given b)p(a \given b) 
\end{equation}
for all $x,a,b$ such that $p(b)>0$.
\end{definition}

The intersection property of conditional independence is defined as follows \citep[e.g.][1.1.5]{Pearl2009}.
\begin{definition}[Intersection Property] \label{def:inters}
We say that the joint distribution of $X,A,B,C$ satisfies the \emph{intersection property} if
\begin{equation} \label{eq:inters}
X \independent A \given B, C \; \text{ and } \; X \independent B \given A, C  \quad \implies \quad X \independent (A,B) \given C\ .
\end{equation}
\end{definition}
The intersection property~\eqref{eq:inters} has been proven to hold for strictly positive densities \citep[e.g.][1.1.5]{Pearl2009}. 
It is also known that the intersection property does not necessarily hold if the joint distribution does not have a density \citep[e.g.][]{Dawid79b}. 
\citet{Dawid80} provides measure-theoretic necessary and sufficient conditions for the intersection property.
In this work we assume the existence of a density (A0) and provide more detailed conditions under which the intersection property holds.

\section{Counter Example} \label{sec:ex}
We now give an example of a distribution that does not satisfy the intersection property~\eqref{eq:inters}. Since the joint distribution is absolutely continuous with respect to the Lebesgue measure, the example shows that the intersection property requires further restrictions on the density apart from its existence.
We will later use the same idea to prove Proposition~\ref{prop:counterex} that shows the necessity of our new condition.
\begin{example} \label{ex:cou}
Consider a structural equation model for random variables $X,A,B$: 
\begin{align*}
A &= N_A\,,\\
B &= A + N_B\,,\\
X &= f(B) + N_X\,,
\end{align*}
with $N_A \sim \mathcal{U}([-2;-1] \cup [1;2])$, $N_B, N_X \sim \mathcal{U}([-0.3;0.3])$
being jointly independent. Let the function $f$ be of the form
$$
f(b) = \left\{ 
\begin{array}{cl}
+10& \text{if } b>0.5\,,\\
0& \text{if } b<-0.5\,,\\
g(b)& \text{ else,}
\end{array} \right.
$$
where the function $g$ can be chosen to make $f$ arbitrarily smooth. Some parts of this structural equation model are summarized in Figure~\ref{fig:ex}. We clearly have $X \independent A \given B$ and $X \independent B \given A$ but $X \notindependent A$ and $X \notindependent B$. A formal proof of this statement is provided in the more general setting of Proposition~\ref{prop:counterex}.
It will turn out to be important that the two connected components of the support of $A$ and $B$ cannot be connected by an axis-parallel line. In the notation introduced in Definition~\ref{def:uc} below, this means $Z_1$ and $Z_2$ are not equivalent. 
Within each component, however, that is if we consider the areas $A,B>0$ and $A,B<0$ separately, we do have the independence statement $X \independent A,B$ that is predicted by the intersection property. This observation will be formalized as the weak intersection property in Proposition~\ref{prop:main}. 
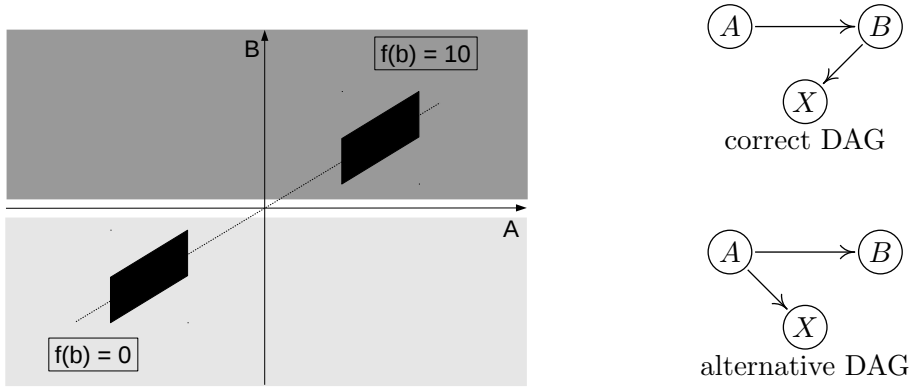
\begin{figure}[ht]
\begin{center}
\includegraphics[width=0.55\textwidth]{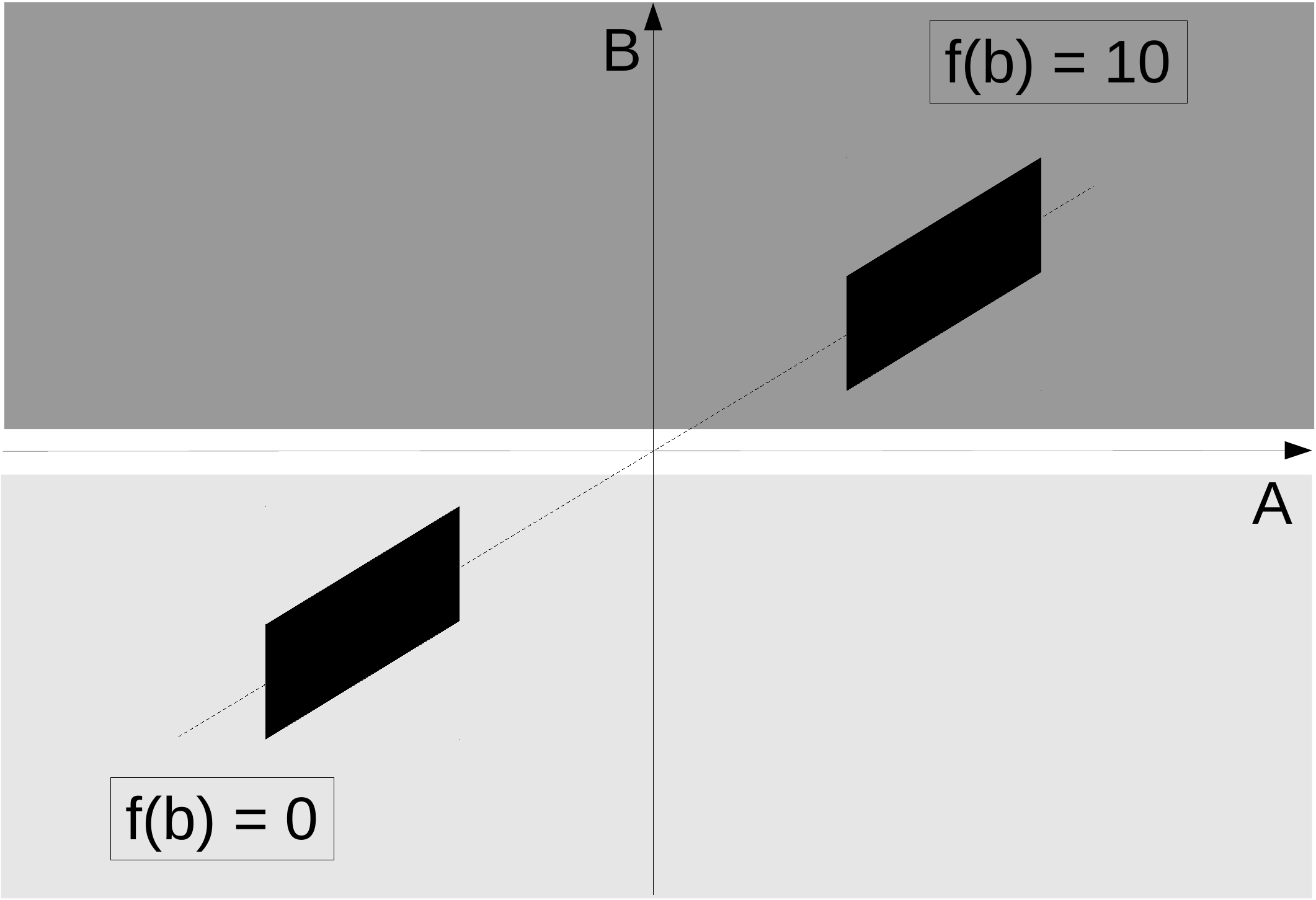}
\hspace{2cm}
\begin{tikzpicture}[scale=1, line width=0.5pt, minimum size=0.58cm, inner sep=0.3mm, shorten >=1pt, shorten <=1pt]
    \normalsize
    \draw (0,1.5) node(a2) [circle, draw] {$A$};
    \draw (2,1.5) node(b2) [circle, draw] {$B$};
    \draw (1,0.5) node(x2) [circle, draw] {$X$};
    \draw[-arcsq] (a2) -- (b2);
    \draw[-arcsq] (a2) -- (x2);
	\draw (1,0) node(xx) {alternative DAG};
    \draw (1,3) node(xx) {correct DAG};
    \draw (0,4.5) node(a) [circle, draw] {$A$};
    \draw (2,4.5) node(b) [circle, draw] {$B$};
    \draw (1,3.5) node(x) [circle, draw] {$X$};
    \draw[-arcsq] (a) -- (b);
    \draw[-arcsq] (b) -- (x);
   \end{tikzpicture}
\end{center}
\caption{Example~\ref{ex:cou}. The plot on the left hand side shows the support of variables $A$ and $B$ in black. In the areas filled with dark gray and light gray the function $f$ takes values ten and minus ten, respectively. The structural equation model corresponds to the top graph but the distribution can also be generated by a structural equation model with the bottom graph.}
\label{fig:ex}
\end{figure}
\end{example}
Example~\ref{ex:cou} has the following important implication for causal inference. The distribution satisfies causal minimality with respect to two different graphs, namely $A \rightarrow B \rightarrow X$ and $X \leftarrow A \rightarrow B$ (see Figure~\ref{fig:ex}). Since it violates faithfulness and the intersection property, we are not aware of any causal inference method that is able to recover the correct graph structure based on observational data only. 
Recall that \citet{Peters2013anm} assume strictly positive densities in order to assure the intersection property. More precisely, the example shows that Lemma~37 in \citep{Peters2013anm} does not hold anymore when the positivity is violated.

\section{Necessary and sufficient condition for the intersection property} \label{sec:main}
This section characterizes the intersection property in terms of the joint density over the corresponding random variables. In particular, we state a weak intersection property (Proposition~\ref{prop:main}) that leads to a necessary and sufficient condition for the classical intersection property, see Corollary~\ref{cor:1}. For these results, the notion of path-connectedness becomes important. A continuous mapping 
$\lambda: [0,1] \rightarrow \mathcal{X}$ into a metric space $\mathcal{X}$ is called a \emph{path} between $\lambda(0)$ and $\lambda(1)$ in $\mathcal{X}$. A subset $\mathcal{S} \subseteq \mathcal{X}$ is called path-connected if every pair of points from $\mathcal{S}$ can be connected by a path in $\mathcal{S}$.
We require the following definition.
\begin{definition} \label{def:uc}
\begin{itemize}
\item[(i)]
For each $c$ with $p(c)>0$ we consider the (not necessarily closed) support of $A$ and $B$:
$$
\mathrm{supp}_c(A,B) := \{(a,b)\ : \ p(a,b,c) > 0 \}\,.
$$ 
We further write for all sets $M \subset \mathcal{A} \times \mathcal{B}$
\begin{align*}
\mathrm{proj}_A(M) &:= \{a \in \mathcal{A}\ : \ \exists b \text{ with } (a,b) \in M \} \; \text{ and}\\
\mathrm{proj}_B(M) &:= \{b \in \mathcal{B}\ : \ \exists a \text{ with } (a,b) \in M \}\,.
\end{align*}
\item[(ii)] We denote the path-connected components of $\mathrm{supp}_c(A,B)$ by $(Z^c_i)_i$. 
Two path-connected components $Z^c_{i_1}$ and $Z^c_{i_2}$ are said to be coordinate-wise connected if 
\begin{align*}
\mathrm{proj}_A (Z^c_{i_1}) \cap \mathrm{proj}_A (Z^c_{i_2}) &\neq \emptyset \qquad \text{ or}\\
\mathrm{proj}_B (Z^c_{i_1}) \cap \mathrm{proj}_B (Z^c_{i_2}) &\neq \emptyset 
\end{align*}
We then say that $Z^c_i$ and $Z^c_j$ are equivalent if and only if there is a sequence 
$Z^c_{i} = Z^c_{i_1}, \ldots, Z^c_{i_m} = Z^c_{j}$ with two neighbours $Z^c_{i_k}$ and $Z^c_{i_{k+1}}$ being coordinate-wise connected.
We represent these equivalence classes by the union of all its members. These unions we denote by $(U^c_{i})_i$.

We further introduce a deterministic function $U^c$ of the variables $A$ and $B$.
We set 
$$
U^c := \left\{ 
\begin{array}{cl}
i & \text{ if } \;(A,B) \in U^c_i\\
0 & \text{ if } \;p(A,B) = 0 
\end{array}
\right..
$$
We have that $U^c = i$ if and only if $A \in \mathrm{proj}_A(U^c_{i})$ if and only if $B \in \mathrm{proj}_B(U^c_{i})$.

Note that the projections $\mathrm{proj}_A(U^c_{i})$ are disjoint (for different $i$); similarly for $\mathrm{proj}_B(U^c_{i})$.
\item[(iii)]
The case where there is no variable $C$ can be treated as if $C$ was deterministic: $p(c) = 1$ for some $c$.
\end{itemize}
\end{definition}
In Example~\ref{ex:cou} there is no variable $C$. Figure~\ref{fig:ex} shows the support $\mathrm{supp}_c(A,B)$ in black. It contains two path-connected components. Since they cannot be connected by axis-parallel lines, they are not equivalent; thus, one of them corresponds to $U_1^C$ and the other to $U_2^c$. Figure~\ref{fig:ex2} shows another example that contains three equivalence classes of path-connected components; again, there is no variable $C$; we  formally introduce a deterministic variable $C$ that always takes the value $c$.
\begin{figure}[ht]
\begin{center}
\includegraphics[width=0.8\textwidth]{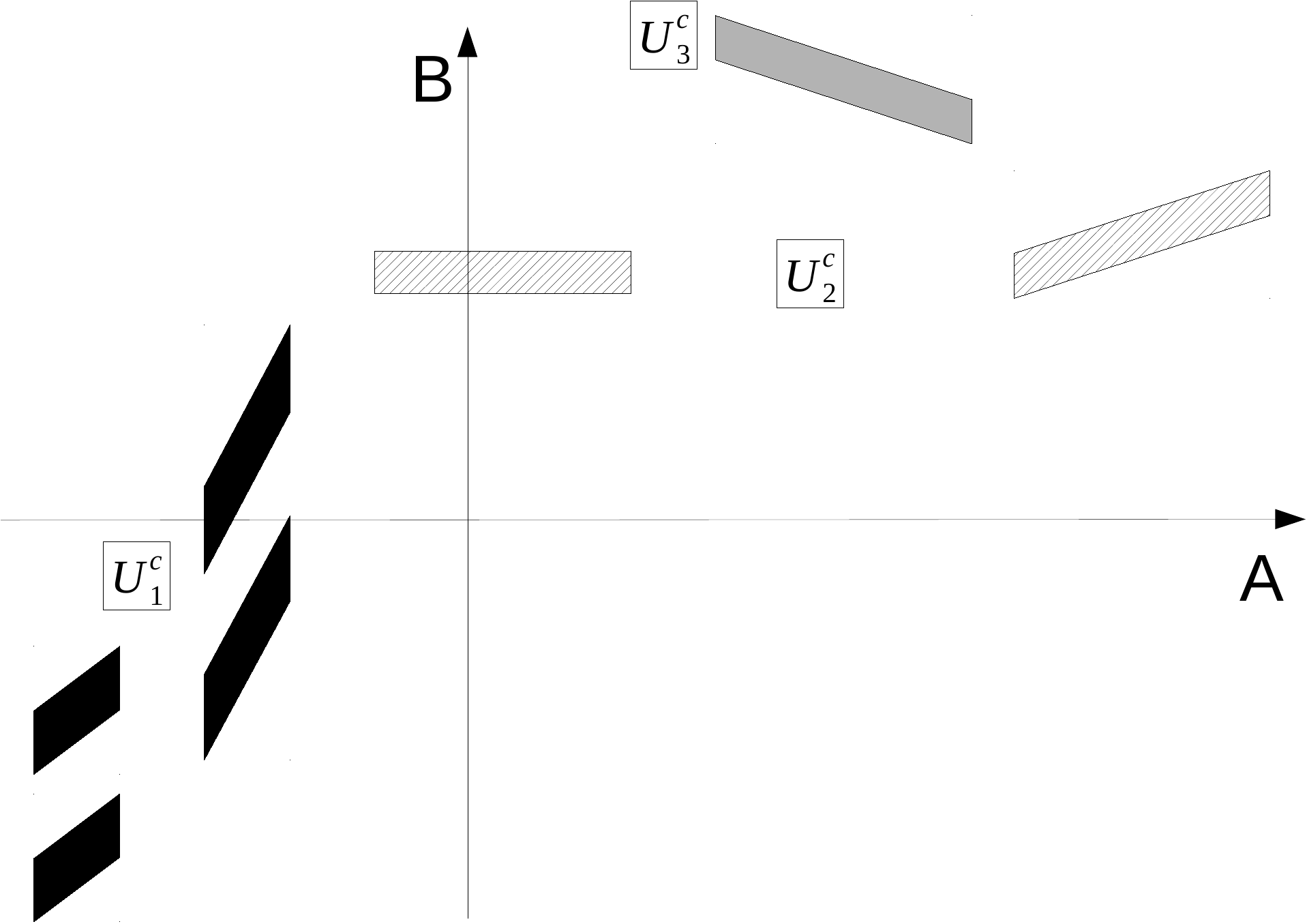}
\end{center}
\caption{Each block represents one path-connected component $Z_i^c$ of the support of $p(a,b)$. All blocks with the same filling are equivalent since they can be connected by axis-parallel lines. There are three different fillings corresponding to the equivalence classes $U_1^c$, $U_2^c$ and $U_3^c$.}
\label{fig:ex2}
\end{figure}

Using Definition~\ref{def:uc} we are now able to state the two main results, Propositions~\ref{prop:main} and~\ref{prop:counterex}. As a direct consequence we obtain Corollary~\ref{cor:1} which generalizes the condition of strictly positive densities.
\begin{proposition}[Weak Intersection Property] \label{prop:main}
Assume (A0), (A1) and that $X \independent A \given B, C$  and  $X \independent B \given A, C$. Consider now $c$ with $p(c)>0$ and the variable $U^c$ as defined in Definition~\ref{def:uc}(ii).
We then have the weak intersection property:
$$
X \independent (A,B) \given C = c, U^c\ .
$$
This means that
$$
p(x\given a,b,c,u^c) = p(x \given c, u^c)
$$
for all $x,a,b$ with $p(a,b,c) > 0$. 
\end{proposition}

\begin{proposition}[Failure of Intersection Property] \label{prop:counterex}
Assume (A0), (A1) and that there are two different sets $U^{c^*}_{1} \neq U^{c^*}_2$ for some $c^*$ with $p(c^*)>0$. Then there is a random variable $X$ such that the intersection property~\eqref{eq:inters} does not hold for the joint distribution of $X,A,B,C$.
\end{proposition}
As a direct corollary from these two propositions we obtain a characterization of the intersection property in the case of continuous densities.
\begin{corollary}[Intersection Property] \label{cor:1}
Assume (A0) and (A1). Then 
\begin{align*}
&\text{The intersection property~\eqref{eq:inters} holds for all variables } X. \\
\Longleftrightarrow \quad &\text{All components } Z^c_i \text{ are equivalent, i.e. there is only one set } U^c_1.
\end{align*}
In particular, this is the case if (A2) holds (there is only one path-connected component) or (A2') holds (the density is strictly positive).
\end{corollary}

\section{Application to Causal Discovery} \label{sec:appl}
We now define what we mean by identifiability of the graph in continuous additive noise models.
Assume that a joint distribution over $X_1, \ldots, X_p$ is generated by a structural equation model (SEM)
\begin{equation} \label{eq:sem}
X_i = f_i(X_{\PA[]{i}}) + N_i\,,
\end{equation}
with continuous, non-constant functions $f_i$, additive and jointly independent noise variables $N_i$ with mean zero and sets $\PA[]{i}$ that are the parents of $i$ in a directed acyclic graph $\mathcal{G}$. To simplify notation, we identify variables $X_i$ with its index (or node) $i$. We consider the following statement
$$
(*) \quad \begin{array}{l}
\mathcal{G} \text{ is identifiable from the joint distribution, i.e. it cannot}\\
\text{ be generated by an SEM with different graph } \mathcal{H} \neq \mathcal{G}\,.
\end{array}
$$
\citet[][Theorem~27]{Peters2013anm} prove this identifiability by extending the identifiability from graphs with two nodes to graphs with an arbitrary number of variables.
Because they require the intersection property, it is shown only for strictly positive densities. But since Corollary~\ref{cor:1} provides weaker assumption for the intersection property, we can use it to obtain new identifiability results.
\begin{proposition} \label{prop:pcnoise}
Assume that a joint distribution over $X_1, \ldots, X_p$ is generated by a structural equation model~\eqref{eq:sem}.
If all densities of $N_1, \ldots, N_p$ are path-connected, then the density of $X_1, \ldots, X_p$ is path-connected, too. Thus, the intersection property~\eqref{eq:inters} holds for any disjoint sets of variables $X,A,B,C \in \{X_1, \ldots, X_p\}$ (see Corollary~\ref{cor:1}). Therefore, statement $(*)$ holds if the noise variables have continuous densities and path-connected support.
\end{proposition}

Example~\ref{ex:cou} violates the assumption of Proposition~\ref{prop:pcnoise} since the support of $A$ is not path-connected. It satisfies another important property, too: the function $f$ is constant on some intervals. The following proposition shows that this is necessary to violate identifiability.
\begin{proposition} \label{prop:nonconstfunct}
Assume that a joint distribution over $X_1, \ldots, X_p$ is generated by a structural equation model~\eqref{eq:sem} with graph $\mathcal{G}$. Let us denote the non-descendants of $X_i$ by $\ND[\mathcal{G}]{i}$.
Assume that the structural equations are non-constant in the following way: for all $X_i$, for all its parents $X_j \in \PA[]{i}$ and for all $X_{\mathbf{C}} \subseteq \ND[\mathcal{G}]{i} \setminus \{X_j\}$, there are $(x_j,x_j',x_k,x_c)$ such that 
$
f_i(x_j,x_k) \neq f_i(x_j',x_k)
$
and $p(x_j,x_k,x_c)>0$ and $p(x_j',x_k,x_c)>0$. Here, $x_k$ represents the value of all parents of $X_i$ except $X_j$. Then for any $\PA[]{i} \setminus \{j\} \subseteq \mathbf{S} \subseteq \ND[\mathcal{G}]{i} \setminus \{j\}$, it holds that
$
X_i \notindependent X_j \, \given \, \mathbf{S}
$.
Therefore, statement $(*)$ follows.
\end{proposition}
Proposition~\ref{prop:nonconstfunct} provides an alternative  way to prove identifiability. The results are summarized in Table~\ref{tab:identifresults}.

\begin{table}
\begin{center}
\begin{tabular}{c|c}
\multirow{1}{*}{additional assumption on continuous ANMs} & identifiability of graph, see $(*)$ \\\hline \hline
\multirow{2}{*}{noise variables with full support}& \cmark \\  
 & \citep{Peters2013anm}\\ \hline 
\multirow{2}{*}{noise variables with path-connected support}& \cmark  \\
 &  Proposition~\ref{prop:pcnoise}\\ \hline
\multirow{2}{*}{non-constant functions, see Proposition~\ref{prop:nonconstfunct}}& \cmark  \\
 & Proposition~\ref{prop:nonconstfunct}\\ \hline
\multirow{2}{*}{none of the above satisfied}& \xmark \\
 & Example~\ref{ex:cou}
\end{tabular}
\caption{This table shows conditions for continuous additive noise models (ANMs) that lead to identifiability of the directed acyclic graph from the joint distributions. Using the characterization of the intersection property we could weaken the condition of a strictly positive density.}
\label{tab:identifresults}
\end{center}
\end{table}

\section{Conclusion}
It is possible to prove the intersection property of conditional independence for variables whose distributions do not have a strictly positive density. A necessary and sufficient condition for the intersection property is that all path-connected components of the support of the density are equivalent, that is they can be connected by axis-parallel lines. In particular, this condition is satisfied for densities whose support is path-connected. In the general case, the intersection property still holds conditioning on any equivalence class of path-connected components, we call this the weak intersection property.

This insight has a direct application in causal inference. For continuous additive noise models we can prove identifiability of the graph from the joint distribution using strictly weaker assumptions than before.

\section{Proofs}

\subsection{Proof of Proposition~\ref{prop:main}}
We require the following well-known lemma \citep[e.g.][]{Dawid79}. 
\begin{lemma} \label{lem:1}
We have $X \independent A \given B$ if and only if
$$
p(x \given a,b) = p(x \given b) 
$$
for all $x,a,b$ such that $p(a,b) > 0$ and $p(b)>0$.
\end{lemma}

\begin{proof}(of Proposition~\ref{prop:main})
We have by Lemma~\ref{lem:1}
\begin{equation} \label{eq:help}
p(x \given b,c) = p(x\given a,b,c) = p(x \given a,c)
\end{equation}
for all $x,a,b,c$ with $p(a,b,c)>0$. 
As the main argument we show that
\begin{equation} \label{eq:aaa}
p(x \given b,c) = p(x \given \tilde b,c)
\end{equation}
for all $x,b,\tilde b, c$ with $b, \tilde b \in \mathrm{proj}_B(U^c_i)$ for the same $i$. \\
{\bf Step 1}, we prove equation~\eqref{eq:aaa} for $b, \tilde b \in Z^c_i$, that is there is a path $(a(t),b(t))$, such that $p(a(t),b(t),c)>0$ for all $0 \leq t \leq 1$, and $b(0) = b$ and $b(1) = \tilde b$. Since the interval $[0,1]$ is compact and $p$ is continuous, the path $\{(a(t),b(t))\ : \ 0 \leq t \leq 1\}$ is compact, too. Define for each point $(a(t),b(t))$ on the path an open ball with radius small enough such that all $(a,b)$ in the ball satisfy $p(a,b,c)>0$. Since this is an open cover of the space, choose a finite subset, of size $n$ say, of all those balls that still provide an open cover of the path. Without loss of generality generality let $(a(0),b(0))$ be the center of ball $1$ and $(a(1),b(1))$ be the center of ball $n$. It suffices to show that equation~\eqref{eq:aaa} holds for the centres of two neighbouring balls, say $(a_1,b_1)$ and $(a_2, b_2)$. Choose one point $(a^*,b^*)$ from the non-empty intersection of those two balls. Since $d((a_1,b_1), (a^*,b_1)) < d((a_1,b_1), (a^*,b^*))$ and 
$d((a_2,b_2), (a_2,b^*)) < d((a_2,b_2), (a^*,b^*))$
for the Euclidean metric $d$, we have that $p(a_1,b_1,c)$, $p(a^*,b_1,c)$, $p(a^*,b^*,c)$, $p(a_2,b^*,c)$ and $p(a_2,b_2,c)$ are all greater zero. Therefore, using equation~\eqref{eq:help} several times, 
\begin{align*}
p(x \given b_1,c) &= p(x \given a_1,c) = p(x \given a^*,c)\\
&= p(x \given b^*,c) = p(x \given a_2,c) =p(x \given b_2,c)
\end{align*}
This shows equation~\eqref{eq:aaa} for $b, \tilde b \in Z^c_i$.\\
{\bf Step 2}, we prove equation~\eqref{eq:aaa} for $b \in Z^c_i$ and $\tilde b \in Z^c_{i+1}$, where $Z^c_i$ and $Z^c_{i+1}$ are coordinate-wise connected (and thus equivalent). If 
$b^* \in \mathrm{proj}_B(Z^c_i) \cap \mathrm{proj}_B(Z^c_{i+1})$, we know that 
$$
p(x \given b,c) = p(x \given b^*,c) = p(x \given \tilde b,c)
$$
from the argument given in step~1 above. If $a^* \in \mathrm{proj}_A(Z^c_i) \cap \mathrm{proj}_A(Z^c_{i+1})$, then there is a $b_i, b_{i+1}$ such that $(a^*,b_i) \in Z^c_i$ and $(a^*,b_{i+1}) \in Z^c_{i+1}$. By equation~\eqref{eq:help} and the argument from step~1 we have
$$
p(x \given b,c) = p(x \given b_i,c) = p(x \given b_{i+1},c) = p(x \given \tilde b,c)
$$
We can now combine these two steps in order to prove the original claim from equation~\eqref{eq:aaa}. If $b, \tilde b \in \mathrm{proj}_B(U^c_i)$ then $b \in \mathrm{proj}_B(Z^c_1)$ and $\tilde b \in \mathrm{proj}_B(Z^c_n)$, say. Further, there is a sequence $Z^c_1, \ldots, Z^c_n$ coordinate-connecting these components. Combining steps 1 and 2 proves equation~\eqref{eq:aaa}.\\

Consider now $x,b,c$ such that $p(b,c)>0$ (which implies $p(c)>0$) and consider $u^c = i$, say. 
Observe further that $p(a,c)>0$ for $a \in \mathrm{proj}_A(U^c_i)$. 
We thus have
\begin{align*}
p(x,u^c \given c) &= \int_{a} p(x,a,u^c \given c) \ da 
= \int_{a \in \mathrm{proj}_A(U^c_i)} p(x,a \given c) \ da\\
&= \int_{a \in \mathrm{proj}_A(U^c_i)} \frac{p(x,a,c) p(a ,c)}{p(c)p(a,c)} \ da \\
&= \int_{a \in \mathrm{proj}_A(U^c_i)} p(x \given a,c) p(a \given c) \ da \\
&= \int_{a \in \mathrm{proj}_A(U^c_i), p(a,b,c)>0} p(x \given a,c) p(a \given c) \ da \\
& \qquad \qquad \qquad \qquad \qquad + \int_{a \in \mathrm{proj}_A(U^c_i),p(a,b,c)=0} p(x \given a,c) p(a \given c) \ da \\
&= p(x \given b,c) \int_{a \in \mathrm{proj}_A(U^c_i), p(a,b,c)>0} p(a \given c) \ da  + \int_{\mathcal{A}_b} p(x \given a,c) p(a \given c) \ da\\
&=: (\#) 
\end{align*}
with $\mathcal{A}_b = \{a \in \mathrm{proj}_A(U^c_i)\ : \ p(a,b,c)=0\}$. It is the case, however, that for all $a \in \mathcal{A}_b$ there is a $\tilde b(a) \in \mathrm{proj}_B(U^c_i)$ with $p(a,\tilde b(a),c)>0$.
But since also $b \in \mathrm{proj}_B(U^c_i)$ we have $p(x\given \tilde b,c) = p(x\given b,c)$ by equation~\eqref{eq:aaa}. Ergo,
\begin{align*}
(\#) &= p(x \given b,c) \int_{a \in \mathrm{proj}_A(U^c_i), p(a,b,c)>0} p(a \given c) \ da  + \int_{\mathcal{A}_b} p(x \given a,\tilde{b}(a), c) p(a \given c) \ da\\
&= p(x \given b,c) \int_{a \in \mathrm{proj}_A(U^c_i),p(a,b,c)>0} p(a \given c) \ da  + p(x \given b, c) \int_{\mathcal{A}_b} p(a \given c) \ da\\
&= p(x \given b, c) \int_{a \in \mathrm{proj}_A(U^c_i)} p(a \given c) \ da\\
&= p(x\given b,c)\ p(u^c\given c)
\end{align*}
This implies 
$$
p(x\given c,u^c) = p(x\given b,c)\ .
$$
Together with equation~\eqref{eq:help} this leads to
$$
p(x \given a,b,c,u^c) = p(x \given a,b,c) = p(x \given c, u^c)\,.
$$

\end{proof}

\subsection{Proof of Proposition~\ref{prop:counterex}}
\begin{proof}
Define $X$ according to 
$$
X = g(C,U^C) + N_X
$$
where $N_X \sim \mathcal{U}([-0.1,0.1])$ is uniformly distributed with $(N_X, A, B, C)$ being jointly independent. Define $g$ according to
$$
g(c,u^c) = \left\{
\begin{array}{cl}
10 & \text{if } C = c^* \text{ and } u^{c^*} = 1\\  
0 & \text{ otherwise} 
\end{array} \right.
$$
Fix a value $c$ with $p(c)>0$.
We then have for all $a, b$ with $p(a, b, c)>0$ that 
$$
p(x \given a, b, c) = p(x \given c, u^c) = p(x \given a, c) = p(x \given b, c) 
$$
because $U^c$ can be written as a function of $A$ or of $B$. We therefore have that $X \independent A \given B, C$ and $X \independent B \given A, C$. Depending on whether $b$ is in $\mathrm{proj}_B(U^{c^*}_1)$ or not we have $p(x=0 \given b,c^*)=0$ or $p(x=10 \given b,c^*)=0$, respectively. Thus,
\begin{eqnarray*}
&p(x=10 \given b,c^*) \cdot p(x=0 \given b,c^*)& = 0\,\text{, whereas}\\
&p(x=10 \given c^*) \cdot p(x=0 \given c^*)& \neq 0\,.
\end{eqnarray*}
This shows that $X \notindependent B \given C = c^*$.
\end{proof}

\subsection{Proof of Proposition~\ref{prop:pcnoise}}
\begin{proof}
Since the true structure corresponds to a directed acyclic graph, we can find a causal ordering, i.e. a permutation $\pi:\{1, \ldots, p\} \rightarrow \{1, \ldots, p\}$ such that
$$
\PA[]{\pi(i)} \subseteq \{\pi(1), \ldots, \pi(i-1)\}\,.
$$
In this ordering, $\pi(1)$ is a source node and $\pi(p)$ is a sink node.
We can then rewrite the structural equation model in~\eqref{eq:sem} as
$$
X_{\pi(i)} = \tilde f_{\pi(i)}(X_{\pi(1), \ldots, \pi(i-1)}) + N_{\pi(i)}\,, 
$$
where the functions $\tilde f_i$ are the same as $f_i$ except they are constant in the additional input arguments. The statement of the proposition then follows by the following argument: consider a one-dimensional random variable $N$ with mean zero and a (possibly multivariate) random vector $X$ both with path-connected support and a continuous function $f$. Then, the support of the random vector $(X,f(X)+N)$ is path-connected, too.
Indeed, consider two points $(x_0,y_0)$ and $(x_1,y_1)$ from the support of $(X,f(X)+N)$. The path can then be constructed by concatenating three sub-paths: (1) the path between $(x_0,y_0)$ and $(x_0,f(x_0))$ ($N$'s support is path-connected), (2) the path between $(x_0,f(x_0))$ and $(x_1,f(x_1))$ on the graph of $f$ (which is path-connected due to the continuity of $f$) and (3) the path between $(x_1,f(x_1))$ and $(x_1,y_1)$, analogously to (1).

Therefore the statements of Lemma~37 and thus Proposition~28 from \citet{Peters2013anm} remain correct, which proves $(*)$ for noise variables with continuous densities and path-connected support.
\end{proof}

\subsection{Proof of Proposition~\ref{prop:nonconstfunct}}
\begin{proof}
The proof is immediate. Since $p(x_i \given x_j,x_k,x_c) \neq 
p(x_i \given x_j',x_k,x_c)$ (the means are not the same) the statement follows from Lemma~\ref{lem:1}.

In this case, Lemma~37 might not hold but more importantly Proposition~28 does \citep[both from][]{Peters2013anm}. This proves $(*)$.
\end{proof}

\bibliographystyle{plainnat} 
\bibliography{bibliography}
\end{document}